\documentclass[reqno,11pt]{amsart}
\usepackage{a4wide,color,eucal,enumerate,mathrsfs}
\usepackage[normalem]{ulem}
\usepackage{amsmath,amssymb,epsfig,amsthm} 
\usepackage[utf8]{inputenc}
\usepackage{psfrag}
\numberwithin{equation}{section}

\newtheorem{theorem}{Theorem}[section]

\newtheorem{lemma}[theorem]{Lemma}

\newtheorem{definition}[theorem]{Definition}

\theoremstyle{remark}

\DeclareMathOperator{\diam}{diam}
\DeclareMathOperator{\dist}{dist}

\newcommand{\N}{\mathbb{N}}

\newcommand{\R}{\mathbb{R}}

\def\dist{{\mathop\mathrm{\,dist\,}}}

\def\bint{{\ifinner\rlap{\bf\kern.35em--}
\int\else\rlap{\bf\kern.45em--}\int\fi}\ignorespaces}

\def\bbint{{\ifinner\rlap{\bf\kern.35em--}
\hspace{0.078cm}\int\else\rlap{\bf\kern.45em--}\int\fi}\ignorespaces}

\def\diam{{\mathop\mathrm{\,diam\,}}}
\def\card{{\mathop\mathrm{\,card\,}}}

\def\bint{{\ifinner\rlap{\bf\kern.35em--}
\int\else\rlap{\bf\kern.45em--}\int\fi}\ignorespaces}

\begin{document}

\title[]
{Dimension estimate for the two-sided points of\\
planar Sobolev extension domains}
\author{Jyrki Takanen}

\address{University of Jyvaskyla \\
         Department of Mathematics and Statistics \\
         P.O. Box 35 (MaD) \\
         FI-40014 University of Jyvaskyla \\
         Finland}
\email{jyrki.j.takanen@jyu.fi}

\thanks{The author acknowledges the support from the Academy of Finland, grant no. 314789.}
\subjclass[2000]{Primary 46E35, 28A75.}
\keywords{}
\date{\today}



\begin{abstract}
In this paper we give an estimate for the Hausdorff dimension of the set of two-sided points of the boundary of bounded simply connected Sobolev $W^{1,p}$-extension domain for $1<p<2$. Sharpness of the estimate is shown by examples. We also prove the equivalence of different definitions of two-sided points.
\end{abstract}

\maketitle

\section{Introduction}
This paper is part of the study of the geometry of the boundary of Sobolev extension domains in Euclidean spaces.
Recall that a domain $\Omega$ is a $W^{1,p}$-extension domain if there exists a bounded operator
$E\colon W^{1,p}(\Omega) \to W^{1,p}(\R^n)$
with the property that $Eu|_\Omega = u$ for each $u \in W^{1,p}(\Omega)$. 
Here, for $1\leq p\leq \infty$, we denote by $W^{1,p}(\Omega)$ the set of all functions in $L^p(\Omega)$ whose first distributional derivatives are in $L^p(\Omega)$. The space $W^{1,p}(\Omega)$ is normed by 
$$\Vert u \Vert_{W^{1,p}(\Omega)} := \Vert u \Vert_{L^p(\Omega)}+\Vert\nabla u \Vert_{L^p(\Omega)}.$$

Additionally, in the case $p>1$ the operator may be chosen to be linear (see \cite{HKT08}).
When $p=1$ linearity of the operator is know only for bounded simply connected domains (\cite{KoskelaW11domains}).

Several classes of domains are known to be $W^{1,p}$-extension domains. For instance, Lipschitz domains (Calderón \cite{Calderon61} $1<p<\infty$, and Stein \cite{Stein70} $p=1,\infty$). Jones \cite{Jones1981} introduced a wider class of $(\epsilon,\delta)$-domains and proved that every $(\epsilon,\delta)$-domain is a  $W^{1,p}$-extension domain. 
Notice that the Hausdorff dimension of the boundary of a Lipschitz domain is $n-1$ and the boundary is rectifiable. For an  $(\epsilon,\delta)$-domain the Hausdorff dimension of the boundary may be strictly greater than $n-1$ and it may not be locally rectifiable (for example the Koch snowflake). However, an easy argument shows that the boundary of an $(\epsilon, \delta)$-domain can not self-intersect.

In the case where $\Omega \subset \mathbb R^2$ is bounded and simply connected, the $W^{1,p}$-extendability has been characterized. From the results in \cite{GLV1979}, \cite{GMR1990},  \cite{GV1981}, \cite{Jones1981}, we know that a bounded simply connected domain $\Omega\subset \mathbb{R}^2$ is a $W^{1,2}$-extension domain if and only if $\Omega$ is a quasidisk, or equivalently an uniform domain.

In \cite{Shvartsman2010} Shvartsman proved the following characterization for $W^{1,p}$-extension domains. For $2<p<\infty$ and $\Omega$ a finitely connected bounded planar domain, then $\Omega$ is a Sobolev $W^{1,p}$-extension domain if and only if for some $C>1$ the following condition is satisfied: for every $x,y \in \Omega$ there exists a rectifiable curve $\gamma \subset \Omega$ joining $x$ to $y$ such that
$$\int_\gamma \dist (z,\partial \Omega)^{\frac{1}{1-p}} \, {\rm d}s(z) \leq C \Vert x -y \Vert ^{\frac{p-2}{p-1}}.$$
In particular, when $2 \le p < \infty$, a finitely connected bounded $W^{1,p}$-extension domain $\Omega$ is quasiconvex, meaning that there exists a constant $C\geq 1$ such that any pair of points in $z_1,z_2 \in \Omega$ can be connected with a rectifiable curve $\gamma \subset \Omega$ whose length satisfies $\ell(\gamma ) \leq C |z_1-z_2|$.

In paper \cite{KoskelaW1pChar} the case $1 < p < 2$ was characterized: a bounded simply connected $\Omega \subset \mathbb{R}^2$ is a Sobolev $W^{1,p}$-extension domain if and only if there exists a constant $C > 1$ such that for every $z_1,z_2 \in \R^2 \setminus {\Omega}$ there exists a curve  $\gamma \subset \R^2 \setminus {\Omega}$ connecting $z_1$ and $z_2$ and satisfying
 \begin{equation}\label{eq:curve_condition_intro}
  \int_{\gamma} \dist(z,\partial \Omega)^{1-p}\,{\rm d}s(z) \leq C\|z_1-z_2\|^{2-p}.
 \end{equation}

The above geometric characterizations give bounds for the size of the boundary of Sobolev extension domains.
The following estimate for the Hausdorff dimension of the boundary for simply connected $W^{1,p}$-extension domain $\Omega$ in the case $p\in(1,2)$ was given in \cite{LRT2020} :
$$\dim_{\mathcal H}(\partial \Omega) \leq 2- \frac{M}{C},$$ where $C$ is the constant in \eqref{eq:curve_condition_intro} and $M>0$ is an universal constant.
Recall that for $s > 0$, the \emph{$s$-dimensional Hausdorff measure} of a subset $A \subset \mathbb{R}^n$ is defined by
$$\mathcal H^s(A) = \lim_{\delta \downarrow 0} \mathcal H_\delta^s (A),$$
where 
$\mathcal H _\delta^s (A) = \inf \left\{ \sum_i \diam(E_i)^s : A \subset \bigcup_i E_i, \diam (E_i) \leq \delta \right\}.$
The \emph{Hausdorff dimension of a set $A \subset \mathbb{R}^n$} is then given by
$$\dim_{\mathcal H} (A) = \inf \{t : \mathcal H^t(A)< \infty\}.$$

In this paper, we are interested in the case $1<p<2$, when the boundary of $\Omega$ may self-intersect, (for examples see \cite[Example 2.5]{Koskela1990}, \cite{Deheuvels2016}, and Section \ref{sec:sharpness}). More accurately, we study the size of the set of two-sided points. In the case of $2 \leq p \leq \infty$, there are no such points which can be seen from the quasiconvexity.  The case $p=1$ has been characterized in \cite{KoskelaW11domains} as a variant of quasiconvexity of the complement. 
In this case the dimension of the set of two-sided points does not depend on the constant in quasiconvexity. 
Let us define what we mean by a two-sided point. Here we give a definition which generalizes to $\mathbb{R}^n$, but the proof of our main theorem will use an equivalent formulation based on conformal maps, see Section \ref{sec:equivalence}.
 \begin{definition}[Two-sided points of the boundary of a domain]\label{def:two-sidedpoint}
 Let $\Omega \subset \mathbb{R}^n$ be a domain. A point $x \in \partial\Omega$ is called \emph{two-sided}, if there exists $R>0$ such that for all $r \in (0,R)$ there exist connected components $\Omega_r^1$ and $\Omega_r^2$ of $\Omega \cap B(x,r)$ that are nested: $\Omega_r^i \subset \Omega_s^i$ for $0 < r < s < R$ and $i \in \{1,2\}$.
\end{definition}
We denote by $\mathcal T$ the two-sided points of  $\Omega$. Note, that the nestedness condition in Definition \ref{def:two-sidedpoint} for the connected components $\Omega_r^i$ implies that $x \in \partial \Omega_r^i$.
We establish the following dimension estimate for $\mathcal T$ for simply connected planar $W^{1,p}$-extension domains.
\begin{theorem}\label{thm:main}
Let $1 < p < 2$ and $\Omega \subset \mathbb{R}^2$ a simply connected, bounded Sobolev $W^{1,p}$-extension domain. Let $\mathcal T$ be the set of two-sided points of $\Omega$. Then 

\begin{equation}\label{eq:mainclaim}
\dim_{\mathcal H}(\mathcal T) \leq 2- p + \log _{2} \left(1- \frac{2^{p-1}-1}{2^{5-2p}C}\right) , 
\end{equation} 

where $C$ is the constant in \eqref{eq:curve_condition_intro}. 
\end{theorem}

Recalling that necessarily $C \ge 1$, the estimate \eqref{eq:mainclaim} implies the existence of a constant $M_1(p)>0$ depending only on $p$, such that
 \[
 \dim_{\mathcal H}(\mathcal T) \leq 2- p - \frac{M_1(p)}{C}. 
 \]
 In Section \ref{sec:sharpness} we show the existence of another constant $M_2>0$ and examples $\Omega_{p,C}$ of Sobolev $W^{1,p}$-extension domains 
 for every $1 < p < 2$ and every $C> C(p)$ for which
 \begin{equation}\label{eq:sharpnessintro}
 \dim_{\mathcal H}(\mathcal T_{\Omega_{p,C}}) \geq 2- p - \frac{M_2}{C}. 
 \end{equation}
 This shows the sharpness of Theorem \ref{thm:main}.

\section{Equivalent definitions for two-sided points}\label{sec:equivalence}

In this section we give equivalent conditions for the set of two-sided points in the case that the domain is John. We note that a bounded simply connected planar domain satisfying the condition \eqref{eq:curve_condition_intro} is John {(\cite[Chapter 6 Thm 3.5]{GMR1990})}. Recall, that $\Omega$ is a $J$-John domain, if there exists a constant $J > 0$ and a point $x_0 \in \Omega$  so that for every $x \in \Omega$ there exists a unit speed curve $\gamma\colon [0,\ell(\gamma)]\to \Omega$ such that $\gamma(0) = x$, $\gamma(\ell(\gamma)) = x_0$, and 
 \begin{equation}\label{eq:John}
 \dist(\gamma(t), \partial \Omega) \geq Jt \qquad \text{for all }t \in [0,\ell(\gamma)].
 \end{equation}
 
We denote the open unit disk of the plane by $\mathbb {D}$.
  For a bounded simply connected John domain $\Omega \subset \mathbb R^2$, a conformal map $f\colon \mathbb D \to \Omega$ can always be extended continuously to a map $f \colon \overline{\mathbb D} \to \overline{\Omega}$.
  This is because a
  John domain is finitely connected along its boundary \cite{NakkiVaisala1991} and a conformal map from the unit disk to $\Omega$ can be extended continuously onto the closure $\overline{\Omega}$ if and only if the domain is finitely connected along its boundary \cite{Palka1991}.

\begin{theorem}\label{thm:equivalence}
Let $\Omega\subset \mathbb R^2$ be a bounded simply connected John domain (especially, if $\Omega$ is a bounded simply connected $W^{1,p}$-extension domain for $1 < p < 2$). 
Let $f\colon \mathbb D \to \Omega$ be a conformal map extended continuously to a function $\overline{\mathbb D} \to \overline{\Omega}$ still denoted by $f$. Define
$$E = \{x \in \partial \Omega : f^{-1} (\{x\}) \text{ disconnects } \partial \mathbb D \}$$
and
$$\tilde E = \{ x \in \partial \Omega :  \, \card( f^{-1}(\{x\})) >1\}. $$
Then
\[
\mathcal T = E = \tilde E,
\]
where $\mathcal T$ is the set of two-sided points according to Definition \ref{def:two-sidedpoint}.
\end{theorem}

In the proof of Theorem \ref{thm:equivalence} we need the following lemma.

\begin{lemma}\label{lma:componentsubsets}
Let $\Omega \subset \mathbb{R}^2$ be a simply connected John domain, let $x \in \partial\Omega$, and $r \in (0,\diam(\Omega))$. If there exist two disjoint open sets $U_1,U_2 \subset \Omega \cap B(x,r)$ such that $x \in \partial U_1\cap \partial U_2$ and both of the sets $U_1$ and $U_2$ are unions of connected components of $\Omega \cap B(x,r)$. Then there exist connected components $U_1'$ and $U_2'$ of $U_1$ and $U_2$ respectively, such that $x \in \partial U_1'\cap\partial U_2'$.  
\end{lemma}
\begin{proof}
Let us first show that there exists $N \in \N$ such that 
\begin{equation}\label{eq:finitecomponents}
    \card{\{ \tilde{\Omega} 
    \,:\, \tilde{\Omega} \text{ connected component of }\Omega \cap B(x,r) \text{ such that } \tilde{\Omega} \cap B(x,r/2) \neq \emptyset\}} \leq N.
\end{equation}
Take $M\in \mathbb{N}$ components $\tilde{\Omega}_i$ as in \eqref{eq:finitecomponents}, and choose from each one a point $x_i \in \tilde{\Omega}_i\cap B(x,r/2)$. Let $\gamma_i$ be a John curve connecting $x_i$ to a fixed John center $x_0$ of $\Omega$. 
For each $i$ for which $x_0 \notin \tilde \Omega_i$, the curve $\gamma_i$ must exit $B(x,2r/3)$.
For these $i$ we consider points $y_i \in \gamma_i \cap S(x,2r/3)$, 
which then exist for all but maybe one of the indexes $i$. 
By the John condition there exists balls $B_i = B(y_i, Jr/6) \subset \tilde{\Omega}_i$.
As the balls $B_i$ are disjointed and $B_i$ covers  an arc of $S(x,2r/3)$ of length at least $J r/3$, we have $ (M-1) Jr/3   \leq \frac{4}{3} \pi r$, hence $M-1 \leq (\frac{Jr}{4\pi} ))^{-1}$.

Next we show that \eqref{eq:finitecomponents} implies the claim of the lemma. Define
\[
\{A_j\}_{j=1}^k := \{ \tilde{\Omega} \subset U_1 : \tilde{\Omega} \text{ connected component of }\Omega \cap B(x,r) \text{ such that } \tilde{\Omega} \cap B(x,r/2) \neq \emptyset\}.
\]
By \eqref{eq:finitecomponents} we have $k \le N$.
Since $U_1$ consists of connected components of $\Omega \cap B(x,r)$,
we have
\[
  U_1 \cap B(x,r/2) \subset \bigcup_{j=1}^k A_j.
\]
Now, because $x \in \overline{\bigcup_{j=1}^k  A_j}  = \bigcup_{j=1}^k \overline{A_j}$ there exists $j$ such that $x \in \overline{A_j}$. We call this $A_j$ the set $U_1'$. Similarly we find $U_2'$ for $U_2$.
\end{proof}

Notice that Lemma \ref{lma:componentsubsets} does not hold for general simply connected domain  $\Omega$, for example consider the topologist's comb. 

\begin{proof}[Proof of Theorem \ref{thm:equivalence}]
We divide the proof into several claims. Showing that
\[
\tilde{E} \subset E \subset \mathcal T \subset \tilde{E}.
\]

\medskip
\noindent
{\color{blue}\textsc{Claim 1}}: $\tilde{E} \subset E$.

Let $z \in \partial \Omega$ such that $\card (f^{-1} (z)) > 1$, and $ A = \partial\mathbb D \setminus f^{-1}(z)$. Let $x_1,x_2 \in f^{-1}(z)$.
By \cite[Theorem 10.9]{Pommerenke1992}, the set $f^{-1} (z)$
has Hausdorff dimension zero.
Therefore, we find points of $A$ from both components of $\partial\mathbb D \setminus \{x_1,x_2\}$. Hence $A$ is disconnected in $\partial\mathbb D$, and thus $z \in E$.

\medskip
\noindent
{\color{blue}\textsc{Claim 2}}: $\mathcal T \subset \tilde{E}$.

Let $z \in \mathcal{T}$.  By assumption there exists $R>0$ such that for each $0<r<R$ there exists disjoint  connected components $\Omega_r^1,\Omega_r^2 \subset \Omega\cap B(z,r)$, with the property that $\Omega_r^i\subset \Omega_s^i$ when $0<r<s$. Towards a contradiction, assume that $f^{-1}(z)$ is a singleton ($w=f^{-1}(z)$). By continuity of $f$ (up to the boundary) there exists $\varepsilon>0$ such that $f(B(w,\varepsilon) \cap \overline{\mathbb{D}}) \subset B(z,r)$. As a continuous image of a  connected set $f(B(w,\varepsilon)\cap \mathbb{D})$ is connected. We show that $f^{-1}(\Omega_{r}^j)\cap B(w,\varepsilon)\ne \emptyset$ for $j = 1,2$, which gives a contradiction with $\Omega_r^j$ being the disjoint connected components of $B(z,r)\cap \Omega$. Let $(z_i^j)_{i=1}^\infty \subset \Omega_r^j$ be a sequence such that $z_i^j \to z$. By going to a subsequence, we may assume that $(f^{-1}(z_i^j))_{i=1}^\infty$ converges to a point $w^j \in \overline{f^{-1}(\Omega_r^j)}$. Since $f$ is continuous, $f(w^j) = z$. But then $w^j = w$ by the uniqueness of the preimage of $z$. Hence, $f^{-1}(z_i^j) \to w$ meaning that for some $i$ we have $f^{-1}(z_i^j) \in B(w,\varepsilon)$ showing $f^{-1}(\Omega_{r}^j)\cap B(w,\varepsilon)\ne \emptyset$. Therefore, $\Omega_r^j \cap f(B(w,\varepsilon) \cap \mathbb{D}) \neq \emptyset$, connecting sets $\Omega_r^j$. This completes the proof.

\medskip
\noindent
{\color{blue}\textsc{Claim 3}}: $E \subset \mathcal T$.

 Let $z \in E$. We will show that $z \in \mathcal T$. We do this by first showing by induction that there exists $i_0 \in \mathbb N$ so that for all $i \ge i_0$ there exist connected components $\Omega_{2^{-i}}^j$ of $\Omega \cap B(z,2^{-i})$, $j \in \{1,2\}$, that are nested for fixed $j \in \{1,2\}$.
 At each step of the induction we will have to make sure that $z \in \partial \Omega_{2^{-i}}^1 \cap \partial \Omega_{2^{-i}}^2$.

\medskip

{\color{blue}\textsc{Initial step}}:
Let us show that there exists $r>0$ such that $B(z,r)\cap \Omega$ may be written as union of two disjointed open sets such that $z$ is contained in the boundary of both sets.
First, since $f^{-1} (z) = \cap_{r>0} f^{-1}(B(z,r)\cap \partial \Omega)$, there exists $R>0$ such that $H=f^{-1}(B(z,R)\cap \partial \Omega)$ is disconnected in $\partial \mathbb{D}$. By the continuity of $f$, $K=f^{-1}(\bar B (z,R/2))$ is a closed set in the closed disk $\bar {\mathbb{D}}$. Let 
$y_1,y_2 \in \partial \mathbb{D} \setminus H$ such that $y_1$ and $y_2$ are in different connected components of  $\partial \mathbb{D} \setminus H$. Define $e = \min (\dist (y_1,K), \dist(y_2,K))/2$. Now $K \setminus B(0,1-e)$ is disconnected in $\bar{\mathbb{D}}$. Next we notice that $\dist(f(\bar B(0,1-e)) , \partial \Omega) = R' >0$. Thus the original claim holds with the radius $r=\min(R,R')/2$. Let us now define $i_0 \in \mathbb N$ to be the smallest integer for which $2^{-i_0} \le r$. Call $U_1$ and $U_2$ the two disjoint open sets for which $z \in \partial U_1 \cap \partial U_2$ and 
$\Omega \cap B(x,2^{-i_0}) = U_1 \cup U_2$. By Lemma \ref{lma:componentsubsets} we have connected components $\Omega_{2^{-i_0}}^1 \subset U_1$ and  $\Omega_{2^{-i_0}}^2 \subset U_2$ of $\Omega \cap B(z,2^{-i_0})$ such that $z \in \partial \Omega_{2^{-i_0}}^1 \cap \partial \Omega_{2^{-i_0}}^2$.

\medskip

{\color{blue}\textsc{Induction step}}:
Assume that for some $i \in \mathbb{N}$ there exist disjoint connected components $\Omega_{2^{-i}}^1$ and $\Omega_{2^{-i}}^2$ of $\Omega \cap B(z,2^{-i})$ such that $z \in \partial \Omega_{2^{-i}}^1 \cap \partial \Omega_{2^{-i}}^2$. Let $U_1 = \Omega_{2^{-i}}^1 \cap B(z,2^{-i-1}).$ 

Let us show that $U_1$ is some union of connected components of $\Omega \cap B(z,2^{-i-1})$. Let $V$ be a connected component of $U_1$. It suffices to show that $V$ is a connected component of $\Omega \cap B(z,2^{-i-1})$. Take a connected component $V' \supset V $ of $\Omega \cap B(z,2^{-i-1})$. There exists connected component $W'$ of $\Omega \cap B(z,2^{-i})$ such that $W' \supset V'$.
Since $\emptyset \ne V \subset W' \cap  \Omega_{2^{-i}}^1$ we have $W' = \Omega_{2^{-i}}^1$. Furthermore $V' \subset \Omega_{2^{-i}}^1 \cap B(z,2^{-i-1}) = U_1$. As $V'$ is connected we have $V'=V$. 

Similarly for $U_2$. Now, by Lemma \ref{lma:componentsubsets} we may choose connected components $U_1' \subset U_1$ and $U_2'\subset U_2$ (of $\Omega \cap B(z,2^{-i-1})$) such that $z \in \partial U_1' \cap \partial U_2'$.

\medskip

{\color{blue}\textsc{General $r\in(0,2^{-i_0})$}}:
 Let $2^{-i-1} < r < 2^{-i}$. Let $\Omega_r^1$ be the connected component of $\Omega \cap B(z,r)$ containing $\Omega_{2^{-i-1}}^1$. Since $\Omega_{2^{-i}}^1$ is connected component of $\Omega \cap B(z,2^{-i})$ containing $\Omega_{2^{-i-1}}^1$, we have $\Omega_r^1 \subset \Omega_{2^{-i}}^1$. Let us show that $\Omega_r^1 \subset \Omega_s^1$ for all $0 < r < s$. Let $0 < r< s$. We consider two cases: (1) If $2^{-i-1}< r < s < 2^{-i}$ the sets $\Omega_r^1$ and $\Omega_s^1$ are connected components of $\Omega \cap B(z,r)$ and $ \Omega \cap B(z,s)$, respectively, both containing $\Omega _{2^{-i-1}} ^1$. Since $\Omega_r^1 \subset \Omega  \cap B(z,s)$ and $\Omega_r^1$ is connected we have $\Omega_r^1 \subset \Omega_s^1$.

(2) If $2^{-i-1} \leq r \leq 2^{-i}  \leq 2^{-j-1} \leq s \leq 2^{-j}$ sets $\Omega_r^1$ and $\Omega_s^1$ are connected components of  $\Omega\cap B(z,r)$ and $\Omega\cap B(z,s)$ which contain $\Omega_{2^{-i-1}}^1$ and $\Omega_{2^{-j-1}}^1$, respectively. Similarly as in (1) we have $\Omega_r^1 \subset \Omega_{2^{-i}}^1 \subset \cdots \subset \Omega_{2^{-j-1}}^1 \subset \Omega_s^1$. 
\end{proof}

\section{Upper bound for the dimension of the set of two-sided points}
 In this section we prove Theorem \ref{thm:main} that establishes an upper bound for the Hausdorff dimension of the set of two-sided points $\mathcal T$ for a planar simply connected $W^{1,p}$-extension domain when $1 < p < 2$. We do this by using one of the equivalent definitions of two-sided points given in Theorem \ref{thm:equivalence}.
 Namely, we consider
 $$
  E = \{x \in \partial \Omega : f^{-1} (\{x\}) \text{ disconnects } \partial \mathbb D \},
 $$
 where $f \colon \mathbb D \to \Omega$ is a conformal map extended continuously to a function $\overline{\mathbb D} \to \overline{\Omega}$ still denoted by $f$ (see the beginning of Section \ref{sec:equivalence}).

The idea of the proof is to reduce the dimension estimate of $E$ to a dimension estimate of $E$ along a single curve $\gamma$ satisfying \eqref{eq:curve_condition_intro}. This is possible by Lemma \ref{lma:countablecurves}. Then, on each $\gamma$ the dimension estimate is obtained via Lemma \ref{lma:hausdorffdim} by estimating the number of balls needed to cover the set $E \cap \gamma$ at different scales.

Following the ideas of  \cite[Lemma 4.6]{KoskelaW1pChar} we first show that the set of two-sided points can be covered by a countable union of curves fulfilling condition \eqref{eq:curve_condition_intro}.
\begin{lemma}\label{lma:countablecurves}
Let $1<p<2$ and let $\Omega \subset \mathbb R ^2$ be bounded simply connected $W^{1,p}$ Sobolev extension domain. Then
there exists a countable collection $\Gamma$ of curves satisfying \eqref{eq:curve_condition_intro} so that for the set $E$ of two-sided points we have
$$E \subset \bigcup_{\gamma \in\Gamma} \gamma^o \cap \partial \Omega,$$
where $\gamma^o$ denotes the curve $\gamma$ without its endpoints.
\end{lemma}
\begin{proof}
Let  $f: \overline{\mathbb{D}} \to \overline{\Omega}$ be continuous and conformal in $\mathbb{D}$.
	Let $\{x_j\} \subset \partial \mathbb D$ be dense and 
	for each pair $(x_j,x_i)$, $i \ne j$ select a curve $\gamma_{i,j}$ satisfying \eqref{eq:curve_condition_intro} between the points $f(x_i)$ and $f(x_j)$. Define
	$\Gamma= \{\gamma_{i,j}\,:\, i \ne j\}$.
	
	Now, let $z \in E$. By the definition of $E$ there exist $x_a, x_b \in f^{-1}(z)$, $x_a \neq  x_b$ which divide $\partial \mathbb D$ into two components $I_a$ and $I_b$, so that $f(I_a) \ne \{z\} \ne f(I_b)$. By the continuity of $f$, there exist $i, j$, $i \neq j$ such that $x_i \in I_a$ and $x_j \in I_b$ and $f(x_i) \neq z \neq f(x_j)\ne f(x_i)$. 
	 Let $\gamma_{i,j} \in \Gamma$ be the curve connecting $f(x_i) =: z_i$ and $f(x_j) =: z_j$.  Let  $\tilde \gamma := f([x_i, 0 ] \cup [0, x_j])$. The curve $[x_i, 0 ] \cup [0, x_j]$ divides $\mathbb D$ into two components $A$ and $B$. 
By interchanging $A$ and $B$ if necessary, we have $x_a \in \bar A$ and $x_b \in \bar B$, and by continuity $z \in \overline{f(A)} \cap \overline{f(B)}$. 

Since, the curve $\gamma_{i,j}$ may be assumed to be injective (\cite[Lemma 3.1]{Falconer86}), and since $z_i \ne z_j$, 
the curve $\tilde \gamma \cup \gamma_{i,j}$ is Jordan.
 Let $\tilde{A}$ and $\tilde{B}$ be the corresponding Jordan components. Since $f(A) \subset \tilde A$, $f(B) \subset \tilde B$ we have $z \in \overline{\tilde A} \cap \overline{\tilde B} = \gamma_{i,j} \cup \tilde \gamma$.
Furthermore, since $\tilde \gamma \subset f(\mathbb D) \cup \{z_i,z_j\} = \Omega \cup \{z_i,z_j\}$,
we have $z \in \gamma_{i,j}$.
\end{proof}

To prove Theorem \ref{thm:main} we need the following sufficient condition for an upper bound of the Hausdorff dimension.

\begin{lemma} \label{lma:hausdorffdim}
Let $\gamma \colon J \to \mathbb{R}^2$ be a rectifiable curve from a compact interval $J \subset \R$ and $E \subset \gamma(J)$. Let $0 < \lambda < 1$ and $i_0 \in \mathbb{N}$.
Define for each $i \geq i_0$ a maximal $\lambda^i$-separated net 
$$\{x_k^i \} _{k \in I_i} \subset E\cap \gamma(J).$$
Assume that the following holds: For each $i \geq i_0$ and $k \in I_i$ there exists $j > i$, such that
$$N_j < \lambda^{-(j-i) s },$$
where $N_j = \card(\{ l \in I_j : B(x_l^j, \lambda^{j} )\cap B(x_k^i, \lambda^{i}) \neq \emptyset \})$. Then $\dim_{\mathcal H } (E) \leq s$.
\end{lemma}
\begin{proof}
Define $\mathcal B _{i_0} = \{ B(x_k^{i_0}, \lambda^{i_0}) : k \in I_{i_0}\}$ and inductively for $n > i_0$ by 
$$\mathcal B _{n} = \bigcup_{B(x_k^i, \lambda ^i ) \in \mathcal B _{n-1}} \{  B(x_m^j , \lambda ^j ) : B(x_m ^j , \lambda ^j ) \cap B(x_k^i , \lambda ^i) \neq \emptyset \},$$
where $j=j(i,k)>i$ is given by the assumption.
Clearly $\mathcal B _n$ is a cover of $E$ for each $n \geq i_0$, and for all $B \in \mathcal {B}_n$ 
$$\diam (B) \leq 2 \lambda^n.$$
By assumption, for each $B= B(x^{i}_{k},\lambda^i) \in \mathcal B_{n-1}$ and with $j=j(i,k)$ again given by the assumption
$$\sum_{B(x_m^j, \lambda^j) \cap B \neq \emptyset} \diam (B(x_m^j , \lambda ^j ))^s = N_j (2 \lambda^j)^s < (2 \lambda ^i )^s = \diam (B)^s,$$ 
and therefore 
$$\sum_{ B \in \mathcal B _{n}} \diam(B) ^s \leq \sum _{B \in \mathcal B_{n-1}} \diam(B)^s.$$

Let $\delta >0$ and choose $n \in \mathbb N$ such that $2 \lambda ^n < \delta$. Now
\begin{align*}
\mathcal H_\delta^s (E) &\leq \sum_{B \in \mathcal B _{n}} \diam(B)^s \leq 
\sum_{B \in \mathcal B _{n-1}} \diam(B)^s \leq \ldots \\
& \leq \sum_{B \in \mathcal B _{i_0}} \diam(B)^s \leq \card({I_{i_0}})(2\lambda^{i_0})^s < \infty.
\end{align*}
By letting $\delta \to 0$, we get $\mathcal H^s (E) \le \card({I_{i_0}})(2\lambda^{i_0})^s < \infty$, and consequently 
$\dim_{\mathcal H } (E) \leq s$.
\end{proof}

\begin{proof}[Proof of Theorem \ref{thm:main}]
By Theorem \ref{thm:equivalence} we have $E = \mathcal T$. Let $\Gamma$ be the set of curves given in Lemma \ref{lma:countablecurves}. Let $\gamma \in \Gamma$ and
define the set
$$\{x_k^i\}_{k \in I_i} \subset E\cap \gamma$$
to be a maximal $2^{-i}$ separated net for all $i \in \mathbb{N}$. 
Take $s < \dim_ {\mathcal H} (E \cap \gamma)$. Then, by Lemma \ref{lma:hausdorffdim}, there exists $i \in \mathbb{N}$ and $k \in I_{i}$ such that $N_j \geq 2^{(j-i)s}$ for all $j >i$, where 
$$N_j = \card(\{ l \in I_j : B(x_l^j, 2^{-j} )\cap B(x_k^i, 2^{-i}) \neq \emptyset \}).$$ 
Note that, trivially also $N_i\geq 1$.
Denote $B = B(x_{k}^{i}, 2^{-i+1})$.
For all $j > i+1$ the ball $B$ contains at least $N_{j-1}$ pairwise disjoint balls $B(x_l^{j-1},2^{-j})$ centered at $E \cap \gamma$, and so we have
\begin{equation}\label{eq:hausdorffestimate}
    \mathcal H^1 ( \{ z \in \gamma \cap B : d(z, \partial \Omega) < 2^{-j} \} ) \ge N_{j-1} 2^{-j}.
\end{equation}
Using \eqref{eq:curve_condition_intro}, Cavalieri's principle,  \eqref{eq:hausdorffestimate}, and Lemma \ref{lma:hausdorffdim} we estimate

\begin{align*}
C 2^{-(i-2)(2-p)} &\geq \int _{\gamma \cap B} \dist (z, \partial \Omega)^{1-p} \,{\rm d}z \\
& = \int _0^\infty \mathcal H^1 ( \{ z \in \gamma \cap B : d(z, \partial \Omega)^{1-p} > t \} ) \,{\rm d}t \\
&=\int_0^\infty \mathcal H^1 ( \{ z \in \gamma \cap B : d(z, \partial \Omega) < t^{\frac1{1-p}} \} ) \,{\rm d}t \\
& = \sum _{j \in \mathbb{Z} } \int_{2^{-(j-1)(1-p)}}^{2^{-j(1-p)}} \mathcal H^1 ( \{ z \in \gamma \cap B : d(z, \partial \Omega) < t^{\frac1{1-p}} \} ) \,{\rm d}t \\
& \geq \sum _{j =i+1 } \int_{2^{-(j-1)(1-p)}}^{2^{-j(1-p)}} \mathcal H^1 ( \{ z \in \gamma \cap B : d(z, \partial \Omega) < 2^{-j} \} ) \,{\rm d}t \\
&  \geq  \sum_{j =i+1}  2^{-j(1-p)}(1-2^{1-p})  N_{j-1}2^{-j} \\
& \geq \sum_{j =i+1} (2^{p-1}-1)2^{-(j-1)(1-p)} 2^{(j-1-i)s} 2^{-j} ,
\end{align*}
which implies
\begin{equation}
\begin{split}
    C  &\geq(2^{p-1}-1)2^{2p-5} \sum_{j=i+1} ^\infty 2^{(j-i-1) (s+p-2)}   \\
& = (2^{p-1}-1)2^{2p-5} \frac{1}{1- 2^{-(2-(p+s))}}.  \label{equ200}
\end{split}
\end{equation}
A reordering of \eqref{equ200} gives
$$s \leq 2-p + \log _{2} \left(1- \frac{2^{2p-5}(2^{p-1}-1)}{C}\right).$$
Since $s < \dim_{\mathcal H} (E \cap \gamma)$ was arbitrary, we have
\[
\dim_{\mathcal H} (E \cap \gamma)\leq 2-p + \log _{2} \left(1- \frac{2^{2p-5}(2^{p-1}-1)}{C}\right).
\]
Recalling that $\Gamma$ is countable, and that by Lemma \ref{lma:countablecurves}
\[
E \subset \bigcup_{\gamma \in \Gamma} E\cap \gamma,
\]
the claim follows.
\end{proof}

\section{Sharpness of the dimension estimate}\label{sec:sharpness}

In this section we show the sharpness of the estimate given in Theorem 1.2. We do this by constructing a domain whose set of two-sided points contains a Cantor type set.

Let $0<\lambda < 1/2 $. Let $\mathcal C_\lambda$ be the standard Cantor set obtained as the attractor of the iterated function system $\{f_1 = \lambda x , f_2 = \lambda x +1 - \lambda\}$. For later use we fix some notation. Let $I^1_{0}=[0,1]$, and $\tilde{I}^1_{1}:=(\lambda, 1 -\lambda)$ be the first removed interval. 
We denote by $I_j^i$ the $2^j$ closed intervals left after $j$ iterations, and similarly the $2^{j-1}$ removed open intervals by $\tilde{I}_j^i$.  The lengths of the intervals are
    $$| I_j^i |  =\lambda  ^j , \quad i=1, \ldots, 2^j, j=0,1,2, \ldots$$
and 
$$|\tilde{I}_j^i| = (1-2\lambda) \lambda^{j-1} , \quad i=1, \ldots, 2^{j-1},  j=1,2,3,\ldots.$$ Recall that,  $\mathcal C_\lambda$ is of zero $\mathcal H^1$-measure, and $\dim_{\mathcal H} (\mathcal C_\lambda) = \frac{\log 2}{-\log \lambda}$  (see e.g. \cite[p.60--62]{Mattila1999}).
    
Define $$\Omega_\lambda = (-1,1)^2 \setminus \{(x,y) : x\geq 0, |y| \leq d(x, \mathcal C_\lambda) \}.$$ Set $\Omega_\lambda$ is clearly a domain and the set of two-sided points is $\mathcal C_\lambda \setminus \{(0,0)\}$.

\begin{lemma}\label{lma:constantestimate}
The domain $\Omega_\lambda$ above satisfies the curve condition \eqref{eq:curve_condition_intro} for $1 < p < 2+\frac{\log 2}{\log \lambda} $. That is,
for each $x,y \in \Omega_\lambda^c$ there exists rectifiable curve $\gamma \colon [0,l(\gamma)] \to \Omega_\lambda^c$ connecting $x,y$ such that 

\begin{equation}
\int_\gamma \dist(z, \partial \Omega_\lambda) ^{1-p} \,{\rm d}s(z) \leq C(p, \lambda) | x-y| ^{2-p}. \label{eq:lma_curvecondition}
\end{equation}

Moreover, we have the estimate
\[
C(p,\lambda) \le \frac{c}{(2-p)\lambda^{2-p}(1-2 \lambda^{2-p})}
,
\]
where $c$ is an absolute constant.
\end{lemma}
\begin{proof}
We consider three cases: 
(i) Assume first that $x,y \in \mathbb{R}^2\setminus (-1,1)^2$. Define $\gamma$ as a path of minimal length made of at most four line segments $L_k$  with slope $\pm \frac{\pi}{4}$ such that $\gamma \subset \Omega^c$.  Now for each segment $L_k$

$$\int_{L_k} \dist(z, \partial \Omega)^{1-p} \,{\rm d}s (z) \leq 2\int_0^{|L_k|} (t/\sqrt{2})^{1-p} \,{\rm d}t = \frac{2^{\frac{1-p}{2}}}{2-p} | L_k| ^{2-p} \leq \frac{c}{2-p} |x-y|^{2-p}, $$
where $|L_k|$ is the length of the segment $L_k$.

(ii) Assume $x,y \in \Omega^c \cap (-1,1)^2$. Denote $x=(x_1,x_2)$, $y=(y_1,y_2)$, $\tilde{x}=(x_1,0)$ and $\tilde{y}=(y_1,0)$ and define $\gamma= [x,\tilde{x}]\ast [\tilde{x},\tilde{y}] \ast [\tilde{y} , y]$. 
For $[x,\tilde x]$ (and similarly for $[\tilde y , y ]$) by the geometry of the set $\Omega$
\begin{equation} \label{eq1}
\begin{split}
\int _{[x,\tilde{x}]} \dist(z, \partial \Omega) ^{1-p} \,{\rm d}s (z) & \leq  \int _0^{|\tilde{x}-x|} (t/\sqrt 2) ^{1-p} \,{\rm d} t\\
&=  \frac{2^{\frac{1}{2}(p-1)}}{2-p} |x-\tilde{x}|^{2-p} \\
& \leq  \frac{2^{\frac{1}{2}(p-1)}}{2-p} |x-y|^{2-p}.
\end{split}
\end{equation}

For the segment $[\tilde x, \tilde y]$ let $j \in \mathbb N$ be such that 
$$\lambda^{j} < | \tilde{x}- \tilde{y} | \leq \lambda^{j-1}.$$
Now, $[\tilde x, \tilde y]$ intersects at most two of the intervals $I_{j}^i$ and one interval $\tilde{I}_{j}^i$, where $I_{j}^i$ and $\tilde{I}_j^i$ are the closed and open intervals, respectively, related to the $j$th step of the construction of the Cantor set.  
For every $j$ and $i$ we have
\begin{equation}\label{eq:keskimmainen}
\begin{split}
  \int_{\tilde{I} _j^i } \dist (z, \partial \Omega) ^{1-p } \,{\rm d}s(z) & =
2 \int_0^{\frac{1}{2} (1-2\lambda) \lambda^{j-1}}  \left( \frac{t}{\sqrt{2}}\right)^{1-p} \,{\rm d} t  \\
&=\frac{2^{\frac{3}{2}(p-1)}}{2-p}  (1-2\lambda)^{2-p}\lambda^{(j-1)(2-p)}\\
&=\frac{2^{\frac{3}{2}(p-1)}}{2-p}  |\tilde{I}_j^i|^{2-p}.
\end{split}  
\end{equation}

Since $\mathcal C_\lambda \cap I_j^i$ has a zero $\mathcal{H}^1$-measure, by using \eqref{eq:keskimmainen} for all $k>j$, $i=1, \ldots, 2^k$, we get
\begin{align*}
\int_{I_j^i} \dist(z,\partial \Omega)^{1-p} \,{\rm d}s (z) & =  \sum_{k=j}^\infty \sum_{\tilde{I}_{k+1}^l \subset I_j^i}   \int_{\tilde{I}_{k+1}^l} \dist(z, \partial \Omega)^{1-p}\,{\rm d}s(z) \\
& = \sum_{k=j}^{\infty}   2^{k-j} \cdot 2 \int_0^{\frac{1}{2} (1-2\lambda) \lambda^k}  \left( \frac{t}{\sqrt{2}}\right)^{1-p} \,{\rm d} t \\
& =\frac{2^{ \frac{3(p-1)}{2}  }}{2-p}  (1-2\lambda) ^{2-p} \sum_{k=j}^{\infty} 2^{k-j}\left( \lambda ^{p-2} \right) ^{-k}\\
& = \frac{2^{\frac{3}{2}(p-1)}}{2-p} (1-2\lambda) ^{2-p}  \frac{\lambda^{(2-p)j}}{1-2\lambda^{2-p}}\\
&=\frac{2^{\frac{3}{2}(p-1)}}{2-p} \frac{|\tilde{I}_{j+1}^i|^{2-p}}{1-2 |I_1^i|^{2-p}},
\end{align*}
where the last sum converges by the assumption $p< 2-\log_{\frac{1}{\lambda}} 2$.

Therefore, 
\begin{equation}
\begin{split}
\int_{[\tilde x, \tilde y]} \dist(z, \partial \Omega) ^{1-p } \,{\rm d}s(z)&  \leq \frac{2^{\frac{3}{2}(p-1)}}{2-p} \left( |\tilde{I}_j^i|^{2-p}  + 2 \cdot  \frac{|\tilde{I}_{j+1}^i|^{2-p}}{1-2 |I_1^i|^{2-p}} \right) \\
& \leq \frac{2^{\frac{3}{2}(p-1)}}{2-p}   \frac{(1-2 \lambda)^{2-p}}{\lambda^{2-p}(1-2 \lambda^{2-p})}|\tilde{x} - \tilde{y} | ^{2-p} \\
& \leq \frac{2^{\frac{3}{2}(p-1)}}{2-p}   \frac{(1-2 \lambda)^{2-p}}{\lambda^{2-p}(1-2 \lambda^{2-p})}|x-y | ^{2-p}.
\end{split}\label{eq2}
\end{equation}
Combining \eqref{eq1} and \eqref{eq2} we have
\begin{align*}\int_{\gamma}  \dist(z, \partial \Omega) ^{1-p } \,{\rm d}s(z)  & \leq 3\frac{2^{\frac{3}{2}(p-1)}}{2-p}   \frac{(1-2 \lambda)^{2-p}}{\lambda^{2-p}(1-2 \lambda^{2-p})}|x-y | ^{2-p}\\
& \leq \frac{9}{(2-p)\lambda^{2-p}(1-2 \lambda^{2-p})}|x-y | ^{2-p}.
\end{align*}

(iii): Finally, assume that $x \in \Omega ^c \cap (-1,1)^2$ and $y \in \mathbb{R} ^2 \setminus (-1,1)^2$. Connect $x$ and $w=(1,0)$ with $\gamma_1$ as in (ii) and $w$ and $y$ with $\gamma_2$ as in (i). Let us check that the curve $ \gamma$ defined as $\gamma =\gamma_1 \ast \gamma_2$ fulfils \eqref{eq:lma_curvecondition}.
This follows from the fact that $x$ and $y$ can not be close to each other without being close to $z$  i.e. 
$|x-w| < c_1 |x-y|$
and
$|y-w| <c_2 | x-y|$
with absolute constants $c_1,c_2>0$. Thus
\begin{align*}
\int_{\gamma}  \dist(z, \partial \Omega) ^{1-p } \,{\rm d}s(z) &  = \int_{\gamma_1}  \dist(z, \partial \Omega) ^{1-p } \,{\rm d}s(z) + \int_{\gamma_2}  \dist(z, \partial \Omega) ^{1-p } \,{\rm d}s(z)  \\
&\leq  \frac{9}{(2-p)\lambda^{2-p}(1-2 \lambda^{2-p})}|w-y | ^{2-p} +  \frac{c}{2-p} |x-w|^{2-p}  \\
& \leq \frac{c}{(2-p)\lambda^{2-p}(1-2 \lambda^{2-p})}|x-y | ^{2-p}.
\end{align*}
\end{proof}

Using Lemma \ref{lma:constantestimate} we can now show the existence of constants $M_2>0$ and $C(p)>0$ so that \eqref{eq:sharpnessintro} holds for $C \ge C(p)$.

Fix $p \in (1,2)$, and let $M_2 = \frac{8c}{\log 2}$ where $c$ is the absolute constant from Lemma \ref{lma:constantestimate}. 
In order to make estimates, we use the construction for $\lambda
\in [\frac122^{\frac1{p-2}},2^{\frac1{p-2}})$. By Lemma \ref{lma:constantestimate} we know that the domain $\Omega_\lambda$ satisfies the curve condition with the constant
\begin{equation}\label{eq:constandefinition}
\frac{c}{(2-p)\lambda^{2-p}(1-2 \lambda^{2-p})}.
\end{equation}
Let us define 
\[
C(p) = \frac{c}{(2-p)2^{p-3}(1-2^{p-2})}.
\]
Note that, $C(p)$ equals \eqref{eq:constandefinition} with $\lambda = \frac{1}{2}2^{\frac{1}{p-2}}$.
Now, for $C\geq C(p)$, by the continuity of the constant in \eqref{eq:constandefinition} as a function of $\lambda$ and the fact that it tends to infinity as $\lambda \nearrow 2^{\frac1{p-2}}$, there exists $\lambda_C \in [\frac122^{\frac1{p-2}},2^{\frac1{p-2}})$ such that
\[
C = \frac{c}{(2-p)\lambda_C^{2-p}(1-2 \lambda_C^{2-p})}.
\]
We show that
\begin{equation}\label{eq:sharpness}
\dim{\mathcal C}_{\lambda_C}= -\frac{\log 2}{\log \lambda_C} \ge 2 - p - \frac{M_2}{C}.
\end{equation}
By the assumption $\lambda_C\ge \frac122^{\frac1{p-2}}$, we have
\[
C \le \frac{c2^{3-p}}{(2-p)(1-2 \lambda^{2-p})}
\le \frac{4c}{(2-p)(1-2 \lambda^{2-p})}.
\]
In order to see that \eqref{eq:sharpness} holds, we show that
$$ f_p(\lambda) =2- p - \frac{M_2}{4c}(2-p)(1-2\lambda^{2-p}) + \frac{\log 2}{\log \lambda}$$
is non-positive on the interval $[\frac122^{\frac1{p-2}},2^{\frac1{p-2}})$.
This follows from
\begin{align*}
\min_{\lambda \in [\frac122^{\frac1{p-2}},2^{\frac1{p-2}}]} f_p'(\lambda) & \geq  2\frac{M_2}{4c}(2-p)^2(\frac122^{\frac1{p-2}})^{1-p} -\frac{\log2}{2^{-1} 2^{\frac{1}{p-2}} \log^2(2^{\frac{1}{p-2}})}\\
&=  \frac{M_2}{2c}(2-p)^2(2^{\frac{3-p}{p-2}})^{1-p}
-\frac{(p-2)^2}{2^{\frac{3-p}{p-2}} \log 2}\\
& = \frac{(2-p)^2}{2^{\frac{3-p}{p-2}}}\left(\frac{M_2}{2c}(2^{\frac{3-p}{p-2}})^{2-p}
-\frac{1}{\log 2}\right)\\
& \ge \frac{(2-p)^2}{2^{\frac{3-p}{p-2}}}\left(\frac{M_2}{8c}
-\frac{1}{\log 2}\right)
\ge 0,
\end{align*}
and
\[
f_p(\lambda) \le f_p(2^{\frac1{p-2}}) = 0.
\]
Hence, \eqref{eq:sharpness} holds.

\section*{Acknowledgements}
The author thanks his advisor Tapio Rajala for helpful comments and suggestions. The author thanks  Miguel Garc\'ia-Bravo for his comments, suggestions, and corrections, which improved this paper.

\end{document}